\documentclass[letterpaper,12pt]{article}
\usepackage{braket}
\usepackage{diagbox}
\usepackage{pdflscape} 
\usepackage{adjustbox}

\usepackage{tabularx} 
\usepackage{amsmath,amsthm,amssymb}  
\usepackage{amsmath}
\usepackage{amsfonts}
\usepackage{amssymb}
\usepackage{graphicx} 
\usepackage[margin=0.8in,letterpaper]{geometry} 
\usepackage{cite} 
\usepackage[english]{babel}
\usepackage{caption}
\usepackage{float}
\usepackage[table]{xcolor}
\usepackage{array}
\usepackage{algorithm}
\usepackage{algpseudocode}
\usepackage{cases}
\usepackage{titlesec}
\usepackage{tikz}
\usepackage[utf8]{inputenc}
\usepackage[T1]{fontenc}
\usepackage{multirow}
\usepackage{hyperref}

\usepackage{rotating}
\usepackage{csquotes}
 
\newtheorem{theorem}{Theorem}[section]

\newtheorem{proposition}[theorem]{Proposition}

\newtheorem{corollary}[theorem]{Corollary}

\usepackage{microtype}      
\usepackage{lipsum}
\usepackage{nicefrac}
\usepackage{fancyhdr}       

\DeclareMathOperator{\cp}{\,\square\,}

\usepackage[pagewise,mathlines]{lineno}

\author{
El-Mehdi Mehiri$^{a,}$\footnote{Corresponding author. Email: elmehdi.mehiri@emse.fr or  mehiri314@gmail.com}  \and Sandi Klav{\v{z}}ar$^{b,c,d,}$\footnote{Email: sandi.klavzar@fmf.uni-lj.si}
}

\title{\textbf{M-Polynomial of Product Graphs}}

\date{\today}

\begin{document}

\maketitle

\begin{center}
{\small
$^{a}$ Mines Saint-Étienne, CMP, Department of Manufacturing Sciences and Logistics, 
F-13120 Gardanne, France. \\[4pt]
$^{b}$ Faculty of Mathematics and Physics, University of Ljubljana, Slovenia.\\[4pt]
$^{c}$ Institute of Mathematics, Physics and Mechanics, Ljubljana, Slovenia.\\[4pt]
$^{d}$ Faculty of Natural Sciences and Mathematics, University of Maribor, Slovenia.
}
\end{center}

 \begin{abstract}
\noindent 
The M-polynomial provides a unifying framework for a wide class of degree-based topological indices. Despite its structural importance, general methods for computing the M-polynomial under graph constructions remain limited.  In this paper, explicit formulas, and compact ones whenever possible, for the M-polynomial under different graph products whose vertex sets are the Cartesian product of the factors are developed. The products studied are the direct, the Cartesian, the strong, the lexicographic, the symmetric-difference, the disjunction, and the Sierpi\'{n}ski product. The obtained formulas yield a unified structural description of how vertex-degree interactions propagate under graph constructions and extend existing results for degree-based indices at the polynomial level.
\end{abstract}



\section{Introduction}\label{sec:introduction}

Topological indices are numerical graph invariants designed to encode structural
information in a compact algebraic form.
Originally motivated by chemical graph theory, where molecular structures are
modeled by graphs whose vertices represent atoms and edges represent bonds,
these indices have found applications in quantitative structure--property and 
structure--activity relationships (QSPR/QSAR); see 
\cite{gutman2012mathematical,trinajstic2018chemical}.
Among the numerous classes of such invariants, \emph{degree-based} topological indices occupy
a central position due to their simplicity, interpretability, and strong empirical
correlations with a wide range of physicochemical properties \cite{Dearden2017}.

The last decades have witnessed a growing emphasis on unifying frameworks that
encode entire families of topological indices within a single polynomial.
In the distance-based setting, this role is played by the Wiener polynomial
introduced by Sagan, Yeh and Zhang \cite{SaganWienerpolynomial1996},
whose derivative recovers the classical Wiener index.
The closely related Hosoya polynomial, which generalizes the Wiener number,
has also proven to be a powerful tool in studying composite graphs
\cite{STEVANOVIC2001237}.
Similarly, extensions such as the hyper-Wiener index
\cite{KHALIFEH20081402} and the $y$-Wiener index
\cite{HAMZEH20111099} have led to systematic investigations of
graph operations under distance-based invariants.
Parallel developments exist for degree-based indices,
including the Zagreb indices under graph operations
\cite{KHALIFEH2009804}. 

In the degree-based setting, Deutsch and Klavžar introduced in 2015 the
\emph{M-polynomial}~\cite{EmericMploy}. 
This polynomial serves as a unifying framework for expressing many
degree-based topological indices in closed form, much like the Hosoya
polynomial does for distance-based indices~\cite{deutsch2019m}.
Once the bivariate M-polynomial $M(G;x,y)$ is known, a broad spectrum of classical indices---including
the Zagreb, Randi\'{c}, harmonic, symmetric division, forgotten, and inverse sum
indegree indices---can be obtained by applying differential and integral
operators to $M(G;x,y)$ \cite{EmericMploy}. The recent papers~\cite{das2025closed, kumar2025novel} shows how this approach can also be applied to a large variety of Sombor-type indices, while in~\cite{li-2025} the M-polynomial has been used to derive machine learning models for predicting physicochemical properties of antibiotics. Thus, the computation of the M-polynomial for structured graph families becomes a central task.

Despite its importance, only a limited number of general methods for
computing the M-polynomial are currently available.
Early systematic approaches were developed for restricted graph classes
\cite{deutsch2019m,deutsch2023compute},
while numerous recent papers compute explicit expressions for specific
families of graphs, for a selection see~\cite{akhbari2025developing, abubakar2025exploring, 
ishfaq2024topological, CHAMUA2022, kwun2017m, pradeepa-2025}.
However, comparatively less attention has been given to a structural
analysis of how the M-polynomial behaves under general graph
constructions.

In this paper, we consider graph products in which the vertex set is the Cartesian product of the vertex sets of the factors. The products in question are the four standard graph products (the Cartesian, the direct, the strong, and the lexicographic), the symmetric-difference product, the disjunction product, and the Sierpi\'{n}ski product. Classical works on Boolean operations
\cite{HararyBOOLEAN}
and systematic treatments of composite graphs
\cite{YEH1994359,STEVANOVIC2001237}
demonstrate the fundamental role of such constructions in graph theory.
More recently, generalized product-like operations such as the
Sierpi\'{n}ski product \cite{KoviSierpinskiProduct}
have further expanded this structural landscape. For a systematic treatment of graph products, including Boolean product graphs, we refer the reader to the monograph~\cite{hammack2011handbook}, while for the Sierpi\'{n}ski product, in addition to the seminal paper~\cite{KoviSierpinskiProduct}, see also~\cite{henning-2024, tian-2025}.

The motivation of the present work is both structural and computational.
Our goal is to represent the M-polynomial of a product
$G \ast H$ in terms of the M-polynomials of the factors $G$ and $H$
and their degree polynomials.
This objective is natural from a computational standpoint.
If $|V(G)| = n_G$ and $|V(H)| = n_H$, then the product graph
has $n_G n_H$ vertices.
Even moderate factor sizes lead to a quadratic explosion in the number
of vertices, and consequently in the number of edges.
Since the M-polynomial requires enumerating edges according to
the degree types of their endpoints, a direct computation on
$G \ast H$ quickly becomes expensive.
In contrast, computing $M(G;x,y)$ and $M(H;x,y)$ separately is
considerably cheaper.
If one can derive a structural formula expressing
$M(G\ast H;x,y)$ in terms of invariants of the factors,
then the global computation reduces to algebraic manipulations
of smaller polynomials.

Such factorizations parallel classical formulas obtained for distance-based indices under graph operations \cite{behmaram2011some,YEH1994359,STEVANOVIC2001237,KHALIFEH20081402,HAMZEH20111099}. Similar developments have been pursued for Hosoya and Wiener type polynomials under graph operations \cite{mohamadinezhad2010some}, for degree polynomial  and its behavior under products \cite{Jafarpour2022}, as well as the Zagreb indices \cite{KHALIFEH2009804},  the forgotten F-index \cite{AKHTER201770,Nilanjan2016}, and related degree-based descriptors under various constructions \cite{EmericMploy,Ali_2023,Akhter_2016}. Together, these results emphasize the importance of deriving structural formulas at the level of graph operations rather than recomputing invariants directly on large composite graphs.

The main contribution of this paper is the derivation of
explicit and compact formulas for the M-polynomial for the listed seven graph products. Whenever possible, we express the resulting formulas in terms of the M-polynomials and degree polynomials of the factors,
thereby reducing global complexity to local algebra.  

The paper is organized as follows. In Section~\ref{sec:Preliminaries} we first introduce the notation and definitions used throughout the work.  Then the graph products considered are introduced and their degree formulas stated, which serve as the structural foundation for the subsequent derivations. The main results are stated and proved in Section~\ref{sec:formulas}, each product considered being treated in a separate subsection. In Section~\ref{sec:Paths} the obtained formulas are illustrated using the example of products of paths. Section \ref{sec:Conclusion} summarizes the obtained formulas and discusses perspectives for future research.

\section{Preliminaries}
\label{sec:Preliminaries}

Throughout this paper, all graphs are finite, simple, and undirected. 
For a graph $G$, we denote by $V(G)$ and $E(G)$ its vertex set and edge set, respectively. We write $n_G = |V(G)|$ for the order of $G$. The complement of a graph $G$ is denoted by $\overline{G}$. For a vertex $g \in V(G)$, the (open) neighborhood of $g$ is $N_G(g) = \{g' \in V(G):\ gg' \in E(G)\}$. The degree of $g$ is denoted by $\deg_G(g)$, that is, $\deg_G(g) = |N_G(g)|$. The minimum and the maximum degrees of $G$ are denoted by $\delta(G)$ and $\Delta(G)$, respectively. For $i \ge 0$, let 
\[
n_i(G) = |\{g \in V(G):\ \deg_G(g) = i\}|\,,
\]
and for $i \le j$, let 
\[
m_{i,j}(G) = |\{\{g,g'\} \in E(G):\ \{\deg_G(g), \deg_G(g')\} = \{i,j\}\}|\,.
\]
In Section~\ref{sec:Symmetric} we will also use ordered nonadjacent degree-type numbers, which are defined by 
\[
\widehat{m}_{i,j}(G)
=
|\{(g,g') \in V(G)^2:\ gg' \notin E(G),\; (\deg_G(g),\deg_G(g'))=(i,j)\}|\,.
\]
Note that in the definition of $\widehat m_{i,j}(G)$ we allow $g=g'$. 

The {\em degree polynomial} and the {\em M-polynomial} of a graph $G$ are defined respectively  by
\begin{align*}
D_G(t) & = \sum_{i \in \mathcal{{D}}(G)} n_i(G)\, t^i\,, \\
M(G;x,y) & = \sum_{i \le j} m_{i,j}(G) x^i y^j\,.
\end{align*}

\subsection{Graph Products}

We now present the definitions of the graph products considered in this work. The products of our interest are the four standard graph products (the Cartesian, the direct, the strong, and the lexicographic), the disjunction product (OR), the symmetric-difference product (XOR), and the Sierpi\'{n}ski product. If $G$ and $H$ are graphs, and $\ast$ is an arbitrary of the listed products, then we have $V(G\ast H) = V(G)\times V(H)$. The edges in these products are determined in the following way. 

\begin{itemize}
\item In the {\em Cartesian product} $G \cp H$, vertices $(g_1,h_1)$ and $(g_2,h_2)$ are adjacent if $g_1 = g_2$ and $h_1h_2 \in E(H)$, or $h_1 = h_2$ and $g_1g_2 \in E(G)$.
\item In the {\em direct product} $G\times H$, vertices $(g_1,h_1)$ and $(g_2,h_2)$ are adjacent if $g_{1}g_{2}\in E(G)$ and $h_{1}h_{2}\in E(H)$. 
\item In the {\em strong product} $G \boxtimes H$, vertices $(g_1,h_1)$ and $(g_2,h_2)$ are adjacent if $g_1=g_2$ and $h_1h_2\in E(H)$, or $h_1=h_2$ and $g_1g_2\in E(G)$, or $g_1g_2\in E(G)$ and $h_1h_2\in E(H)$.
\item In the {\em lexicographic product} $G[H]$, vertices $(g_1,h_1)$ and $(g_2,h_2)$ are adjacent if $g_1g_2 \in E(G)$, or $g_1=g_2$ and $h_1h_2\in E(H)$.
\item In the {\em symmetric difference product} (XOR) $G \oplus H$, vertices $(g_1,h_1)$ and $(g_2,h_2)$ are adjacent if either $g_1g_2 \in E(G)$ or $h_1h_2 \in E(H)$ (that is, exactly one of these).
\item In the {\em disjunction product} (OR) $G \vee H$, vertices $(g_1,h_1)$ and $(g_2,h_2)$ are adjacent if $g_1g_2 \in E(G)$ or $h_1h_2 \in E(H)$.
\item In the {\em Sierpi\'{n}ski product} $G\otimes_f H$ with respect to $f:V(G)\to V(H)$, the edge set consists of the following two types of edges. 
\begin{itemize}
\item {\em Inner edges}: for every $g\in V(G)$ and every $hh'\in E(H)$ we have $(g,h)(g,h')\in E(G\otimes_f H)$.
\item {\em Connecting edges}: for every $gg'\in E(G)$, we have $(g,f(g'))(g',f(g))\in E(G\otimes_f H)$.
\end{itemize}
\end{itemize}

Figure~\ref{fig:AllGraphs} illustrates the above products for the case $G \cong H \cong P_3$. 
 
\begin{figure}[ht!]
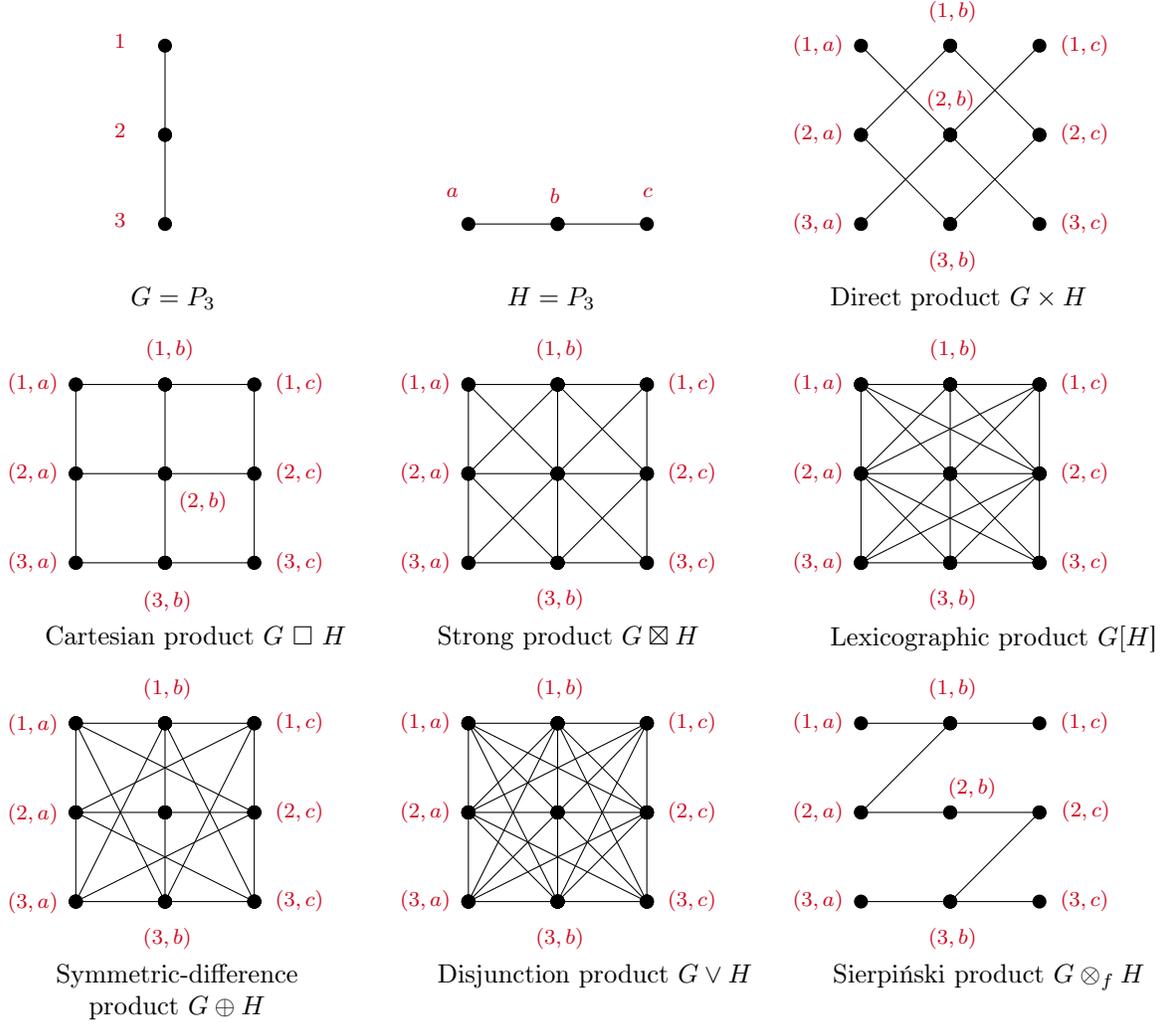

    \centering


\caption{The graph products considered for the case $G\cong H \cong P_3$. In the Sierpi\'{n}ski product $G\otimes_f H$, the function $f$ is defined by $f(1)=a$, $f(2)=b$, and $f(3)=c$.}
\label{fig:AllGraphs}
\end{figure}

Let $G$ and $H$ be graphs, and let $\ast$ be an arbitrary graph product from the seven graph products just introduced. For a given $h\in V(H)$, the set of vertices $\{(g,h):\ g\in V(G)\}$ is called a {\em $G$-layer}. Analogously, for a given vertex $g\in V(G)$, the set of vertices $\{(g,h):\ h\in V(H)\}$ is called an {\em $H$-layer}. When speaking of a $G$-layer (resp.\ $H$-layer) we may also interpret it as the subgraph of $G\ast H$ induced by the vertices from the layer. With this interpretation we may infer that in most of the cases, layers are isomorphic to the corresponding factors. 

In the derivation of the M-polynomial for graph products, the degree of vertices in the resulting graphs play a central role. For convenience and later reference we list degree formulas for the products considered. These formulas are well-known for most of the products considered here, they can be deduced directly from definitions. So let $(g,h)$ be an arbitrary vertex in any of these products. Then we have: 
\begin{align}
\deg_{G\square H}(g,h) & =\deg_G(g) + \deg_H(h)\,, \label{lem:DegreeCartesian} \\
\deg_{G\times H}(g,h) & = \deg_G(g)\deg_H(h)\,, \label{lem:DegreeTensor} \\
\deg_{G\boxtimes H}(g,h) & = \deg_G(g) + \deg_H(h) + \deg_G(g)\deg_H(h)\,, \label{lem:DegreeStrong} \\
\deg_{G[H]}(g,h) & = n_H\deg_G(g)+\deg_H(h)\,, \label{lem:DegreeLexicographic} \\
\deg_{G\oplus H}(g,h) & = n_H \deg_G(g) + n_G\deg_H(h) - 2\deg_G(g)\deg_H(h)\,, \label{lem:DegreeSymmetric} \\
\deg_{G\vee H}(g,h) & = n_H \deg_G(g) + n_G\deg_H(h) - \deg_G(g)\deg_H(h)\,.\label{lem:DegreeDisjunction}
\end{align}
Consider finally the Sierpi\'{n}ski product $G\otimes_f H$ with respect to $f:V(G)\to V(H)$. Then setting 
$$c(g,h) = \bigl|\{g'\in N_G(g):\ f(g')=h\}\bigr|\,,$$
we infer that 
\begin{align}
\deg_{G\otimes_f H}(g,h) & = \deg_H(h) + c(g,h)\,. \label{lem:DegreeSierpinski}
\end{align}

\section{Formulas}
 \label{sec:formulas}

In this section we present formulas for the M-polynomial of all the graph products considered in this paper. Each of the products is discussed in a separate subsection.

\subsection{Cartesian Product}
 \label{sec:Cartesian}

\begin{theorem}\label{thm:cartesian}
If $G$ and $H$ are graphs, then 
\begin{equation}\label{eq:CartesianCompact}
    M(G\cp H;x,y)=D_G(xy)\,M(H;x,y)+D_H(xy)\,M(G;x,y).
\end{equation}

\end{theorem}

\begin{proof}
Let $g\in V(G)$ and let $e = (g,h_1)(g,h_2)$ be an edge from an $H$-layer. Then $h_1h_2\in E(H)$ and by~\eqref{lem:DegreeCartesian}, the endpoints of $e$ have respective degrees $\deg_G(g)+\deg_H(h_1)$ and $\deg_G(g)+\deg_H(h_2)$.
Fix $i,j_1,j_2$ and consider edges $h_1h_2\in E(H)$ with 
    $(\deg_H(h_1),\deg_H(h_2))=(j_1,j_2)$.  
    For each such edge, every choice of $g\in V(G)$ with $\deg_G(g)=i$
    yields an edge in $G\cp H$ whose endpoint degrees are
    $(i+j_1,i+j_2)$.  
    Hence, for each triple $(i,j_1,j_2)$ we obtain
    $n_i(G)\, m_{j_1,j_2}(H)$ edges contributing the monomial 
    $x^{i+j_1}y^{i+j_2}$.  
    Summing over all $i,j_1,j_2$ gives $\sum_{i}\sum_{j_1\le j_2}
n_i(G)\, m_{j_1,j_2}(H)\,
x^{i+j_1}\, y^{i+j_2}$.

By the commutativity of the Cartesian product, we get that the contribution of the edges from $H$-layers is $\sum_{j}\sum_{i_1\le i_2} n_j(H)\, m_{i_1,i_2}(G)\, x^{i_1+j}\, y^{i_2+j}$, hence $ M(G\cp H;x,y)$ is equal to
\begin{equation*}
\sum_{i}\sum_{j_1\le j_2}
n_i(G)\, m_{j_1,j_2}(H)\,
x^{i+j_1}\, y^{i+j_2}
+
\sum_{j}\sum_{i_1\le i_2}
n_j(H)\, m_{i_1,i_2}(G)\,
x^{i_1+j}\, y^{i_2+j}.
\end{equation*}
Using $x^{i+j_1}y^{i+j_2}=(xy)^i x^{j_1}y^{j_2}$, we can factor the first sum as
\[
\sum_i n_i(G)(xy)^i \sum_{j_1\le j_2} m_{j_1,j_2}(H)x^{j_1}y^{j_2}
=
D_G(xy)\,M(H;x,y)\,.
\]
The second sum factors analogously as $D_H(xy)\,M(G;x,y)$.
\end{proof}

\subsection{Direct Product} 
\label{sec:Direct}

 \begin{theorem}
\label{thm:tensor-minmax}
If $G$ and $H$ are graphs, then 
\begin{equation}
\begin{aligned}
M(G\times H;x,y)
=\sum_{i_1\le i_2} \sum_{j_1\le j_2}
m_{i_1,i_2}(G) m_{j_1,j_2}(H) 
&\Bigl( x^{\min\{i_1j_1,  i_2j_2\}}y^{\max\{i_1j_1,  i_2j_2\}}
\\&+
x^{\min\{i_1j_2,  i_2j_1\}}y^{\max\{i_1j_2,  i_2j_1\}}
\Bigr)\,.
\end{aligned}
\end{equation}
\end{theorem}
 
\begin{proof}
Let $g_1g_2\in E(G)$, where $\deg_{G}(g_1)=i_{1} \le i_{2} = \deg_{G}(g_2)$, and let $h_1h_2\in E(H)$, where $\deg_{H}(h_1)=j_{1} \le j_2 = \deg_{H}(h_2)$. By the definition of the direct product, the edges $g_1g_2$ and $h_1h_2$ produce edges $(g_1,h_1)(g_2,h_2)$ and $(g_1,h_2)(g_2,h_1)$
in $G\times H$. By~\eqref{lem:DegreeTensor}, the endpoint degrees of the first one have degrees $(i_{1}j_{1},i_{2}j_{2})$. This edge thus contributes to the degree-type given by the unordered pair
$\{i_{1}j_{1},i_{2}j_{2}\}$, which we encode as $(a_{1},b_{1})$ with $a_{1}\le b_{1}$. Analogously,  the endpoint degrees of the second edge have degrees $(i_{1}j_{2},i_{2}j_{1})$, so it contributes to the degree-type encoded by the unordered pair $\{i_{1}j_{2},i_{2}j_{1}\}$, i.e., $(a_{2},b_{2})$ with $a_{2}\le b_{2}$.

Now fix $(i_{1},i_{2})$ and $(j_{1},j_{2})$.  
There are $m_{i_{1},i_{2}}(G)$ edges in $G$ with endpoint degrees $\{i_{1},i_{2}\}$, and
$m_{j_{1},j_{2}}(H)$ edges in $H$ with endpoint degrees $\{j_{1},j_{2}\}$.  Hence, for fixed $(i_{1},i_{2})$ and $(j_{1},j_{2})$ the total contribution to the
M-polynomial is
\[
m_{i_{1},i_{2}}(G) m_{j_{1},j_{2}}(H)
\bigl( x^{a_{1}}y^{b_{1}} + x^{a_{2}}y^{b_{2}} \bigr)\,,
\]
where $a_{1}=\min\{i_1j_1,  i_2j_2\}$, $b_{1}=\max\{i_1j_1,  i_2j_2\}$, $a_2=\min\{i_1j_2,  i_2j_1\}$, and $b_2=\max\{i_1j_2,  i_2j_1\}$. Summing over all possible degree-types $(i_{1},i_{2})$ of edges in $G$ and
$(j_{1},j_{2})$ of edges in $H$ yields the stated formula.
\end{proof}

Since, $i_1\le i_2$ and $j_1\le j_2$, we can rewrite $M = M(G\times H;x,y)$ as follows
\begin{equation*}
\begin{aligned}
M &=\sum_{i_1\le i_2} \sum_{j_1\le j_2}
m_{i_1,i_2}(G) m_{j_1,j_2}(H) 
\Bigl( x^{i_1j_1}y^{i_2j_2}+
x^{\min\{i_1j_2,  i_2j_1\}}y^{\max\{i_1j_2,  i_2j_1\}}
\Bigr)\\
&=\sum_{i_1\le i_2}m_{i_1,i_2}(G) 
\Bigl( \sum_{j_1\le j_2}
 m_{j_1,j_2}(H) x^{i_1j_1}y^{i_2j_2}+\sum_{j_1\le j_2}
 m_{j_1,j_2}(H) 
x^{\min\{i_1j_2,  i_2j_1\}}y^{\max\{i_1j_2,  i_2j_1\}}
\Bigr)\\
&=\sum_{i_1\le i_2}m_{i_1,i_2}(G) 
\Bigl( M(H;x^{i_1},y^{i_2})+\sum_{j_1\le j_2}
 m_{j_1,j_2}(H) 
x^{\min\{i_1j_2,  i_2j_1\}}y^{\max\{i_1j_2,  i_2j_1\}}
\Bigr)\,.
\end{aligned}
\end{equation*}

\subsection{Strong Product}
 \label{sec:Strong}

To formulate a formula for the M-polynomial of strong products, some preparation is needed. Let $G$ and $H$ be graphs. Then we can partition the set of edges of $G\boxtimes H$ into those that lie in $G$-layers, those that lie in $H$-layers, and the other edges (direct-type ones). We respectively treat them as follows. 

\begin{enumerate}
\item[(1)] For edges $g_1g_2\in E(G)$ and vertices $h\in V(H)$, where $\deg_G(g_1)=i_1$, $\deg_G(g_2)=i_2$, and $\deg_H(h)=j$, set $d_1 = i_1 + j + i_1j$, $d_2 = i_2 + j + i_2j$, and 
\begin{equation}\label{eq:MPolStrongG}
    M^{(G)}(x,y)
=\sum_{j}
\sum_{i_1\le i_2}
n_j(H)  m_{i_1,i_2}(G) 
x^{\min(d_1,d_2)}  
y^{\max(d_1,d_2)}\,.
\end{equation}
    \item[(2)] For edges $h_1h_2\in E(H)$ and vertices $g\in V(G)$, where $\deg_H(h_1)=j_1$, $\deg_H(h_2)=j_2$, and $\deg_G(g)=i$, set $e_1 = i + j_1 + ij_1$, $e_2 = i + j_2 + ij_2$, and 
\begin{equation}\label{eq:MPolStrongH}
    M^{(H)}(x,y)
=\sum_{i}
\sum_{j_1\le j_2}
n_i(G)  m_{j_1,j_2}(H) 
x^{\min(e_1,e_2)}  
y^{\max(e_1,e_2)}\,.
\end{equation}
\item[(3)] For edges $g_1g_2\in E(G)$ and $h_1h_2\in E(H)$, where $\deg_G(g_1)=i_1$, $\deg_G(g_2)=i_2$, $\deg_H(h_1)=j_1$, and $\deg_H(h_2)=j_2$,  with $i_1\le i_2$ and $j_1\le j_2$, set $f_{11} = i_1 + j_1 + i_1j_1$, $f_{12}  = i_1 + j_2 + i_1j_2$, $f_{21}  = i_2 + j_1 + i_2j_1$, $f_{22}  = i_2 + j_2 + i_2j_2$, and 
\begin{equation}\label{eq:MPolStrongTimes}
M^{(\times)}(x,y)
=
\sum_{\substack{i_1\le i_2 \\ j_1\le j_2}}
m_{i_1,i_2}(G)  m_{j_1,j_2}(H) \Big(
x^{\min(f_{11},f_{22})} y^{\max(f_{11},f_{22})}
+
x^{\min(f_{12},f_{21})} y^{\max(f_{12},f_{21})}
\Big)\,.
\end{equation}
\end{enumerate}
 
The above sums are taken over unordered degree pairs, so that $i_1\le i_2$ and $j_1\le j_2$, and the powers of $x$ and $y$ are always ordered using $\min$ and $\max$, in accordance with the M-polynomial convention $i\le j$. With these preparations the main result for the strong product reads as follows. 

\begin{theorem}
\label{thm:strong}
If $G$ and $H$ are graphs, then 
\begin{equation}\label{eq:MPolStrongAll}
    M(G\boxtimes H;x,y)
=
M^{(G)}(x,y) + M^{(H)}(x,y) + M^{(\times)}(x,y)\,.
\end{equation}
\end{theorem}

\begin{proof}
If $(g,h)\in V(G\boxtimes H)$, then by~\eqref{lem:DegreeStrong}, $\deg_{G\boxtimes H}(g,h) = \deg_G(g) + \deg_H(h) + \deg_G(g)\deg_H(h)$. As also already mentioned, we have a partition of $E(G\boxtimes H)$ into three pairwise disjoint sets
\begin{equation}\label{eq:StrongStructure}
    E(G\boxtimes H)
=
E_G \, \dot\cup \, E_H  \, \dot\cup \, E_\times\,,
\end{equation}
where $E_G$ is the set of edges lying in $G$-layers, $E_H$ is the set of edges from $H$-layers, and $E_\times$ is the set of the other edges (the direct-type ones). We examine now the contribution of each of the three parts to the M-polynomial. 

Consider an edge $g_1g_2\in E(G)$ with $\deg_G(g_1)=i_1$ and  $ \deg_G(g_2)=i_2$ where $i_1\le i_2$,   and a vertex $h\in V(H)$ with $\deg_H(h)=j$. Then, using~\eqref{lem:DegreeStrong}, the edge $(g_1,h)(g_2,h)$ contributes the unordered degree-pair $\{\min(d_1,d_2),\max(d_1,d_2)\}$, where $d_1 = i_1 + j + i_1j$ and $d_2 = i_2 + j + i_2j$. There are $m_{i_1,i_2}(G)$ edges $g_1g_2$ in $G$ whose endpoint degrees are $\{i_1,i_2\}$, and there are $n_j(H)$ vertices $h$ in $H$ with degree $j$. Therefore, the number of edges in $E_G$ corresponding to the triple $(i_1,i_2,j)$ is $m_{i_1,i_2}(G) n_j(H)$, and each of them contributes the same monomial $x^{\min(d_1,d_2)}y^{\max(d_1,d_2)}$ in the M-polynomial. Summing over all $i_1\le i_2$ and all degrees $j$ of $H$, we obtain the contribution~\eqref{eq:MPolStrongG} to $M^{(G)}(x,y)$.

By the commutativity of the strong product, we analogously infer that the edges from $H$-layers contribute~\eqref{eq:MPolStrongH} to $M^{(H)}(x,y)$.
 
Consider next the edges from $E_\times$. Let $g_1g_2\in E(G)$ and $h_1h_2\in E(H)$, where $\deg_G(g_1)=i_1$ and $\deg_G(g_2)=i_2$ with $i_1\le i_2$, and $\deg_H(h_1)=j_1$ and $\deg_H(h_2)=j_2$ with $j_1\le j_2$. By~\eqref{lem:DegreeStrong}, 
\[
\begin{aligned}
\deg(g_1,h_1) &= i_1 + j_1 + i_1j_1 = f_{11},\quad \deg(g_1,h_2) = i_1 + j_2 + i_1j_2 = f_{12},\\
\deg(g_2,h_1) &= i_2 + j_1 + i_2j_1 = f_{21}, \quad \deg(g_2,h_2) = i_2 + j_2 + i_2j_2 = f_{22}.
\end{aligned}
\]
Thus, the edge $(g_1,h_1)(g_2,h_2)$ has endpoint degrees          $(f_{11},f_{22})$ and contributes the unordered pair $\{\min(f_{11},f_{22}),\max(f_{11},f_{22})\}$, while the edge $(g_1,h_2)(g_2,h_1)$ has endpoint degrees $(f_{12},f_{21})$ and contributes the unordered pair $\{\min(f_{12},f_{21}),\max(f_{12},f_{21})\}$.

For fixed $i_1\le i_2$ and $j_1\le j_2$, there are $m_{i_1,i_2}(G)$ choices of $g_1g_2\in E(G)$ with endpoint degrees $\{i_1,i_2\}$, and there are $m_{j_1,j_2}(H)$ choices of $h_1h_2\in E(H)$ with endpoint degrees $\{j_1,j_2\}$. Hence the total contribution of the edges from $E_\times$ is
\[
\begin{aligned}
M^{(\times)}(x,y)
&=\sum_{i_1\le i_2}\sum_{j_1\le j_2}
\Big[
m_{i_1,i_2}(G)  m_{j_1,j_2}(H) 
x^{\min(f_{11},f_{22})} y^{\max(f_{11},f_{22})}
\\
&\quad\quad\quad
+
m_{i_1,i_2}(G)  m_{j_1,j_2}(H) 
x^{\min(f_{12},f_{21})} y^{\max(f_{12},f_{21})}
\Big]
\\
&=
\sum_{i_1\le i_2}\sum_{j_1\le j_2}
m_{i_1,i_2}(G)  m_{j_1,j_2}(H) \Big(
x^{\min(f_{11},f_{22})} y^{\max(f_{11},f_{22})}
+
x^{\min(f_{12},f_{21})} y^{\max(f_{12},f_{21})}
\Big)\,.
\end{aligned}
\]
Combining the contributions from $E_G$, $E_H$, and $E_\times$, the result follows. 
\end{proof}

The terms $M^{(G)}(x,y)$, $ M^{(H)}(x,y)$, and $ M^{(\times)}(x,y)$ can be formulated also in a more compact form as we demonstrate next. Since $i_1\leq i_2$ and $j_1\leq j_2$, we have $i_1 + j + i_1j = d_1\le d_2 = i_2 + j + i_2j$, $i + j_1 + ij_1 = e_1\le e_2 = i + j_2 + ij_2$, and $i_1 + j_1 + i_1j_1 = f_{11}\le f_{22} = i_2 + j_2 + i_2j_2$. Therefore,  we can write

\begin{align*}
        M^{(G)}(x,y)
&=\sum_{j}
\sum_{i_1\le i_2}
n_j(H)  m_{i_1,i_2}(G) 
x^{\min(d_1,d_2)}  
y^{\max(d_1,d_2)}\\
&=\sum_{j}n_j(H)
\sum_{i_1\le i_2}
  m_{i_1,i_2}(G) 
x^{d_1}  
y^{d_2}\\
&=\sum_{j}n_j(H)
\sum_{i_1\le i_2}
  m_{i_1,i_2}(G) 
x^{i_1 + j + i_1j}  
y^{i_2 + j + i_2j}\\
&=\sum_{j}n_j(H)x^{j}y^{j}
\sum_{i_1\le i_2}
  m_{i_1,i_2}(G) 
x^{i_1(1 +j)}  
y^{i_2(1 +  j)}\\
&=\sum_{j}n_j(H)x^{j}y^{j}
M(G;x^{1+j},y^{1+j}).
\end{align*}
Similarly, 
\begin{align*} 
    M^{(H)}(x,y)
&=\sum_{i}n_i(G) x^{i}y^{i}M(H;x^{1+i},y^{1+i}).
\end{align*}
Moreover, setting $Z = M^{(\times)}(x,y)$, we can compute as follows: 
\begin{align*}
    Z&=\sum_{i_1\le i_2}\sum_{j_1\le j_2}
m_{i_1,i_2}(G)  m_{j_1,j_2}(H) \Big(
x^{\min(f_{11},f_{22})} y^{\max(f_{11},f_{22})}
+
x^{\min(f_{12},f_{21})} y^{\max(f_{12},f_{21})}
\Big)\\
&=\sum_{i_1\le i_2}m_{i_1,i_2}(G)\sum_{j_1\le j_2}
  m_{j_1,j_2}(H) \Big(
x^{f_{11}} y^{f_{22}}
+
x^{\min(f_{12},f_{21})} y^{\max(f_{12},f_{21})}
\Big)\\
&=\sum_{i_1\le i_2} m_{i_1,i_2}(G) \Big(\sum_{j_1\le j_2}
  m_{j_1,j_2}(H)
x^{ i_1 + j_1 + i_1j_1} y^{i_2 + j_2 + i_2j_2}
+
\sum_{j_1\le j_2}
  m_{j_1,j_2}(H) x^{\min(f_{12},f_{21})} y^{\max(f_{12},f_{21})}
\Big)\\
&=\sum_{i_1\le i_2} m_{i_1,i_2}(G) \Big(x^{i_1}y^{i_2}\sum_{j_1\le j_2}
  m_{j_1,j_2}(H)
x^{  j_1(1 + i_1)} y^{ j_2(1 + i_2)}
+
\sum_{j_1\le j_2}
  m_{j_1,j_2}(H) x^{\min(f_{12},f_{21})} y^{\max(f_{12},f_{21})}
\Big)\\
&=\sum_{i_1\le i_2} m_{i_1,i_2}(G) \Big(x^{i_1}y^{i_2}M(H;x^{1+i_1},y^{1+i_2})
+
\sum_{j_1\le j_2}
  m_{j_1,j_2}(H) x^{\min(f_{12},f_{21})} y^{\max(f_{12},f_{21})}
\Big).
\end{align*}

The Cartesian degree identity~\eqref{lem:DegreeCartesian} can be obtained from the strong-product degree identity~\eqref{lem:DegreeStrong} by dropping the multiplicative term, hence $M(G\cp H;x,y)$ can be recovered from $M(G\boxtimes H;x,y)$ as follows. 

\begin{corollary}
\label{cor:cartesian}
If $G$ and $H$ are graphs, then 
\begin{equation*}
    M(G\cp H;x,y)
=
M^{(G)}_{\square}(x,y)
+
M^{(H)}_{\square}(x,y),
\end{equation*}
where
\begin{equation*}
    M^{(G)}_{\square}(x,y)
=\sum_{j}
\sum_{i_1\le i_2}
n_j(H) m_{i_1,i_2}(G) 
x^{\min(i_1+j,i_2+j)} y^{\max(i_1+j,i_2+j)}\,,
\end{equation*}
\begin{equation*}
    M^{(H)}_{\square}(x,y)
=\sum_{i}
\sum_{j_1\le j_2}
n_i(G) m_{j_1,j_2}(H) 
x^{\min(j_1+i,j_2+i)} y^{\max(j_1+i,j_2+i)}\,.
\end{equation*}
\end{corollary}

\begin{proof}
Start with the proof of Theorem~\ref{thm:strong} and remove the direct adjacency condition. 
Vertex degrees now satisfy \eqref{lem:DegreeCartesian} with no multiplicative term.  
Substituting this degree into the M-polynomial summation formulas yields exactly the two
expressions displayed above. No direct edges appear, so there is no $M^{(\times)}$ term.  
\end{proof}

In contrast, the direct product $G\times H$ cannot be obtained by
a direct specialization of Theorem~\ref{thm:strong}.  Although $G\times H$
uses only the direct adjacency mechanism (i.e., only the edge class
$E_\times$ survives), the vertex degrees in the direct product satisfy the
different identity \eqref{lem:DegreeTensor} rather than  \eqref{lem:DegreeStrong}.
Since the M-polynomial depends on the degrees of the endpoints in the
final graph, one must recompute the endpoint degrees under the
direct-product degree rule; therefore, a direct-product corollary does not
follow from Theorem~\ref{thm:strong} by merely discarding terms.

\subsection{Lexicographic Product}
 \label{sec:Lexicographic}

\begin{theorem}
\label{thm:lexico}
If $G$ and $H$ are graphs, then 
\begin{equation}
    \begin{aligned}
M(G[H];x,y)
&=
\sum_{i}\sum_{j_1\le j_2}
n_i(G)\, m_{j_1,j_2}(H)\,
x^{\,j_1 + i \cdot n_H}\,
y^{\,j_2 + i \cdot n_H}  \\
&\quad+
\sum_{i_1\le i_2}\sum_{j_1 \leq j_2}
m_{i_1,i_2}(G)\, n_{j_1}(H)\, n_{j_2}(H)\,
x^{\,j_1 + i_1 \cdot n_H}\,
y^{\,j_2 + i_2 \cdot n_H}\,.
\end{aligned}
\end{equation}
\end{theorem}

\begin{proof}
Consider first the edges from $G$-layers. Fix a vertex $g\in G$, where $i=\deg_G(g)$.  
Each edge $h_1h_2\in E(H)$ with degrees $(j_1,j_2)$ gives the edge $(g,h_1)(g,h_2)$ whose endpoint degrees are $(j_1 + i \cdot n_H,  j_2 + i \cdot n_H)$. Hence the edges from the $H$-layers altogether contribute
\[
\sum_{i}\sum_{j_1\le j_2}
n_i(G)\,m_{j_1,j_2}(H)\,
x^{j_1+i \cdot n_H} y^{j_2+i \cdot n_H}\,.
\]
The other edges of $G[H]$ connect vertices from different $H$-layers. Let $g_1g_2\in E(G)$, where $i_1 = \deg_G(g_1)$ and $i_2 = \deg_G(g_2)$, and let $h_1, h_2\in V(H)$. Then the degrees of the end-vertices of the edge $(g_1,h_1)(g_2,h_2)$ are $j_1+i_1 \cdot n_H$ and $ j_2 + i_2 \cdot n_H$. Hence the contribution of such edges is 
\[
\sum_{i_1\le i_2}\sum_{j_1\le j_2}
m_{i_1,i_2}(G)\, n_{j_1}(H)n_{j_2}(H)\,
x^{j_1+i_1 \cdot n_H} y^{j_2+i_2 \cdot n_H}.
\]
Adding the two contributions yields the stated formula.
\end{proof}

The M-polynomial of lexicographic products can also be expressed in the following compact form.

\begin{corollary}
\label{cor:lexico-compact}
If $G$ and $H$ are graphs, then 
\begin{equation}\label{eq:LexicographicCompact}
    M(G[H];x,y)
=
M(H;x,y)\, D_G\bigl((xy)^{n_H}\bigr)
\;+\;
M\bigl(G; x^{n_H}, y^{n_H}\bigr)\, D_H(x)\, D_H(y)\,  .
\end{equation}
\end{corollary}

\begin{proof}
We have
\[
\begin{aligned}
&\sum_{i}\sum_{j_1\le j_2}
n_i(G)\, m_{j_1,j_2}(H)\,
x^{j_1+i\cdot n_H} y^{j_2+i\cdot n_H} \\
&\qquad=
\sum_{i} n_i(G) (xy)^{i\cdot n_H}
\sum_{j_1\le j_2} m_{j_1,j_2}(H)\, x^{j_1} y^{j_2} \\
&\qquad=
\Bigl(\sum_{j_1\le j_2} m_{j_1,j_2}(H)\, x^{j_1} y^{j_2}\Bigr)
\Bigl(\sum_i n_i(G)\, (xy)^{i\cdot n_H}\Bigr) \\
&\qquad=
M(H;x,y)\, D_G\bigl((xy)^{n_H}\bigr).
\end{aligned}
\]
Similarly,
\[
\begin{aligned}
&\sum_{i_1\le i_2}\sum_{j_1\le j_2}
m_{i_1,i_2}(G)\, n_{j_1}(H)n_{j_2}(H)\,
x^{j_1+i_1 n_H} y^{j_2+i_2 n_H} \\
&\qquad=
\sum_{i_1\le i_2} m_{i_1,i_2}(G)\,
x^{i_1 n_H} y^{i_2 n_H}
\sum_{j_1} n_{j_1}(H) x^{j_1}
\sum_{j_2} n_{j_2}(H) y^{j_2} \\
&\qquad=
\Bigl(\sum_{i_1\le i_2} m_{i_1,i_2}(G)\,
(x^{n_H})^{i_1} (y^{n_H})^{i_2}\Bigr)
\Bigl(\sum_{j_1} n_{j_1}(H) x^{j_1}\Bigr)
\Bigl(\sum_{j_2} n_{j_2}(H) y^{j_2}\Bigr) \\
&\qquad=
M\bigl(G;x^{n_H},y^{n_H}\bigr)\, D_H(x)\, D_H(y).
\end{aligned}
\]
\medskip
Adding the two contributions yields the claimed formula.
\end{proof}

\subsection{Symmetric-Difference Product}
 \label{sec:Symmetric}

Consider next the symmetric-difference product $G\oplus H$ of graphs $G$ and $H$. If $g\in V(G)$ and $h\in V(H)$, where $\deg_G(g)=i$ and $\deg_H(h)=j$, then by~\eqref{lem:DegreeSymmetric}, 
\begin{equation}
    \deg_{G\oplus H}(g,h) = i(n_H-j)+j(n_G-i)= i\cdot n_H+j\cdot n_G-2ij.
\end{equation}
Since this value is independent of the selection of a vertex of $G$ of degree $i$ and a vertex of $H$ of degree $j$, we set 
$$\delta_{\oplus}^{G\oplus H}(i,j)\;=\;\deg_{G\oplus H}(g,h)\,.$$
If $G$ and $H$ are clear from the context, we may simplify the notation $\delta_{\oplus}^{G\oplus H}(i,j)$ to $\delta_{\oplus}(i,j)$. 

\begin{theorem}
\label{thm:MP-XOR}
If $G$ and $H$ are graphs, then 
\begin{equation}\label{eq:MPolySymmetric}
    \begin{aligned}
M(G\oplus H;x,y)
={}&
\sum_{ i\le i' }\ \sum_{j,j'}
m_{i,i'}(G) \widehat m_{j,j'}(H) 
x^{\min\{\delta_{\oplus}(i,j), \delta_{\oplus}(i',j')\}}
y^{\max\{\delta_{\oplus}(i,j), \delta_{\oplus}(i',j')\}}
\\
+&
\sum_{j\le j'}\ \sum_{i,i'}
m_{j,j'}(H) \widehat m_{i,i'}(G) 
x^{\min\{\delta_{\oplus}(i,j), \delta_{\oplus}(i',j')\}}
y^{\max\{\delta_{\oplus}(i,j), \delta_{\oplus}(i',j')\}}\,.
\end{aligned}
\end{equation}
\end{theorem}

\begin{proof}
In $G\oplus H$, two distinct vertices $(g,h)$ and $(g',h')$ are adjacent if exactly one of $gg'\in E(G)$ and $hh'\in E(H)$ holds. 

Assume first that $gg'\in E(G)$ and $hh'\notin E(H)$, where the endpoints of the edge $gg'$ have degrees $(i,i')$, and the endpoints of $hh'$ have degrees $(j,j')$. The endpoints of the edge $(g,h)(g',h')$ of $G\oplus H$ have degrees $\delta_{\oplus}(i,j)$ and $\delta_{\oplus}(i',j')$. The number of such ordered pairs $(h,h')$ is $\widehat m_{j,j'}(H)$, and there are $m_{i,i'}(G)$ choices for $gg'$. This gives the first summation of~\eqref{eq:MPolySymmetric}. By the symmetry, the second summation is obtained for the case when $gg'\notin E(G)$ and $hh'\in E(H)$. Summing $x^{\min}y^{\max}$ over all edges gives the stated
formula.
\end{proof}

We next demonstrate that  $\widehat m_{i,j}(G)$ can be expressed in terms of the classical coefficients of the M-polynomial of the complement. This relation allows a reformulation of Theorem~\ref{thm:MP-XOR} purely in terms of classical degree-type numbers. To state the reformulation, we first define the indicator function by
\[
\mathbf{1}_{\{i=j\}} :=
\begin{cases}
1; & i=j,\\
0; & i\neq j.
\end{cases}
\]

\begin{proposition}
\label{prop:ordered-nonedges-complement}
If $G$ is a graph, and $i$ and $j$ are degrees of some vertices of $G$, then
\begin{equation}\label{eq:mijComplementExpression}
    \widehat m_{i,j}(G)
=
2 m_{ n_G-1-i,\;n_G-1-j}(\overline G)
+\mathbf 1_{\{i=j\}} n_i(G).
\end{equation}
\end{proposition}

\begin{proof}
Fix $i$ and $j$,  and consider the set
\[
S_{i,j}=\{(g_1,g_2)\in V(G)^2:\ g_1g_2\notin E(G),\ (\deg_G(g_1),\deg_G(g_2))=(i,j)\}.
\]
Split $S_{i,j}$ into disjoint subsets $S_{i,j}^{\neq}=\{(g_1,g_2)\in S_{i,j}: g_1\neq g_2\}$ and $S_{i,j}^{=}=\{(g_1,g_2)\in S_{i,j}\}$.

Consider first the case when $g_1\neq g_2$. For $(g_1,g_2)\in S_{i,j}^{\neq}$, the condition $g_1g_2\notin E(G)$ is equivalent to
$g_1g_2\in E(\overline G)$. Moreover, for every vertex $g_1$,
\[
\deg_{\overline G}(g_1)=n_G-1-\deg_G(g_1),
\]
so $(\deg_G(g_1),\deg_G(g_2))=(i,j)$ is equivalent to
$(\deg_{\overline G}(g_1),\deg_{\overline G}(g_2))=(n_G-1-i,\;n_G-1-j)$.
Hence unordered edges $\{g_1,g_2\}\in E(\overline G)$ of degree-type
$\{n_G-1-i,\;n_G-1-j\}$ correspond bijectively to pairs of ordered
elements $(g_1,g_2)$ and $(g_2,g_1)$ in $S_{i,j}^{\neq}$. Therefore,
\[
|S_{i,j}^{\neq}|=2 m_{ n_G-1-i,\;n_G-1-j}(\overline G).
\]

Consider second the case $g_1=g_2$. We have $(g_1,g_2)\in S_{i,j}^{=}$ if and only if $i=j=\deg_G(g_1)$.
Thus $|S_{i,j}^{=}|=\mathbf 1_{\{i=j\}} n_i(G)$.
 
Summing the two cases yields \eqref{eq:mijComplementExpression} as claimed.
\end{proof}

\subsection{Disjunction Product}
 \label{sec:Disjunction}

In the disjunction product $G\vee H$ of graphs $G$ and $H$, if $g\in V(G)$ has $\deg_G(g)=i$, and $h\in V(H)$ has $\deg_H(h)=j$, then by~\eqref{lem:DegreeDisjunction}, 
\begin{equation*}
    \deg_{G\vee H}(g,h) = i\cdot n_H+j\cdot n_G - ij.
\end{equation*}
We set 
$$\delta_{\vee}^{G\vee H}(i,j)\;=\;\deg_{G\vee H}(g,h)\,.$$
and when $G$ and $H$ are clear from the context, we may simplify the notation $\delta_{\vee}^{G\vee H}(i,j)$ to $\delta_{\vee}(i,j)$. 

\begin{theorem}
\label{prop:M-disjunction}
If $G$ and $H$ are graphs, then 
\begin{equation}\label{eq:MPolyDisjunction}
    \begin{aligned}
M(G\vee H;x,y)
={}&
\sum_{i\le i'}\ \sum_{j,j'}
m_{i,i'}(G) n_j(H) n_{j'}(H) 
x^{\min\{\delta_\vee(i,j), \delta_\vee(i',j')\}}
y^{\max\{\delta_\vee(i,j), \delta_\vee(i',j')\}}
\\
&+
\sum_{j\le j'}\ \sum_{i,i'}
m_{j,j'}(H) n_i(G) n_{i'}(G) 
x^{\min\{\delta_\vee(i,j), \delta_\vee(i',j')\}}
y^{\max\{\delta_\vee(i,j), \delta_\vee(i',j')\}}
\\
&-
\sum_{i\le i'}\ \sum_{j\le j'}
m_{i,i'}(G) m_{j,j'}(H) 
\Bigl(
x^{\min\{\delta_\vee(i,j), \delta_\vee(i',j')\}}
y^{\max\{\delta_\vee(i,j), \delta_\vee(i',j')\}}
\\
&\hspace{50mm}+
x^{\min\{\delta_\vee(i,j'), \delta_\vee(i',j)\}}
y^{\max\{\delta_\vee(i,j'), \delta_\vee(i',j)\}}
\Bigr).
\end{aligned}
\end{equation}
\end{theorem}

\begin{proof}
By the definition of $G\vee H$ we have $E(G\vee H)=E_A\cup E_B$, 
where
\begin{align*}
E_A & = \bigl\{\{(g,h),(g',h')\}:\ gg'\in E(G)\bigr\}\,, \\
E_B & = \bigl\{\{(g,h),(g',h')\}:\ hh'\in E(H)\bigr\}\,.
\end{align*}
Moreover,
\[
E_A\cap E_B
=\bigl\{\{(g,h),(g',h')\}:\ gg'\in E(G), hh'\in E(H)\bigr\}\,.
\]
Fix now an edge $gg'\in E(G)$ whose endpoint degrees in $G$ are $(i,i')$ (with $i\le i'$).
In $E_A$, the second coordinates are arbitrary, so for any choice of
$h,h'\in V(H)$ with $\deg_H(h)=j$ and $\deg_H(h')=j'$, the unordered pair
$\{(g,h),(g',h')\}$ is an edge of $E_A$.
The number of such choices is $n_j(H) n_{j'}(H)$.
For each such edge, the endpoint degrees in $G\vee H$ are
$\delta_\vee(i,j)$ and $\delta_\vee(i',j')$, hence its M-monomial is
\[
x^{\min\{\delta_\vee(i,j),\delta_\vee(i',j')\}}
y^{\max\{\delta_\vee(i,j),\delta_\vee(i',j')\}}.
\]
Summing over all edges of $G$ of type $(i,i')$, of which there are $m_{i,i'}(G)$,
and over all $(j,j')$, gives exactly the first double sum in the statement, which is
$M_A=\sum_{e\in E_A}x^{\min}y^{\max}$. By the symmetry, fixing an edge $hh'\in E(H)$ yields the second sum. 

Consider next the contribution of the intersection $E_A\cap E_B$. 
Fix an edge $gg'\in E(G)$ of type $(i,i')$ and an edge $hh'\in E(H)$ of type $(j,j')$.
There are \emph{exactly two} distinct edges in $E_A\cap E_B$ generated by these two edges:
\[
\{(g,h),(g',h')\}
\quad\text{and}\quad
\{(g,h'),(g',h)\}.
\]
(They are distinct because $h\neq h'$ and $g\neq g'$.)
Their endpoint degree pairs in $G\vee H$ are
\[
\bigl(\delta_\vee(i,j),\delta_\vee(i',j')\bigr)
\quad\text{and}\quad
\bigl(\delta_\vee(i,j'),\delta_\vee(i',j)\bigr),
\]
hence their $M$-monomials are the two terms inside the parentheses in the third sum.
There are $m_{i,i'}(G)m_{j,j'}(H)$ ways to pick the underlying edges $gg'$ and $hh'$,
so summing over all $(i,i')$ and $(j,j')$ yields
\[
M_I=\sum_{e\in E_A\cap E_B}x^{\min}y^{\max}\,,
\]
which is equal to the third (subtracted) sum in the statement.

Since $E(G\vee H)=E_A\cup E_B$ and $E_A,E_B$ are finite,
\[
\sum_{e\in E(G\vee H)}x^{\min}y^{\max}
=
\sum_{e\in E_A}x^{\min}y^{\max}
+
\sum_{e\in E_B}x^{\min}y^{\max}
-
\sum_{e\in E_A\cap E_B}x^{\min}y^{\max},
\]
that is, $M(G\vee H;x,y)=M_A+M_B-M_I$, which is precisely the formula claimed.
\qedhere
\end{proof}

\subsection{Sierpi\'{n}ski Product}
\label{sec:Sierpi}

\begin{theorem}
\label{thm:sierpinski-M}
If $G$ and $H$ are graphs and $f:V(G)\to V(H)$, then 
\begin{equation}
    M(G\otimes_f H;x,y)=M^{(\mathrm{inner})}(x,y)+M^{(\mathrm{conn})}(x,y)\,,
\end{equation}
where
\begin{align}
    M^{(\mathrm{inner})}(x,y)
& =
\sum_{g\in V(G)}\ \sum_{\{h,h'\}\in E(H)}
x^{\min\{d_{g,h},\,d_{g,h'}\}}\,
y^{\max\{d_{g,h},\,d_{g,h'}\}}\,,\\
M^{(\mathrm{conn})}(x,y)
& =
\sum_{\{g,g'\}\in E(G)}
x^{\min\{d_{g,f(g')},\,d_{g',f(g)}\}}\,
y^{\max\{d_{g,f(g')},\,d_{g',f(g)}\}},
\end{align}
and $d_{u,v}=\deg_{G\otimes_f H}(u,v)$ for all $(u,v)\in V(G\otimes_f H)$.
\end{theorem}

\begin{proof}
By the definition of the Sierpi\'{n}ski product, $E(G\otimes_f H)$ is the disjoint union of $E_{\mathrm{inner}}$ and $E_{\mathrm{conn}}$, where 
\begin{align*}
    E_{\mathrm{inner}}
&=\bigl\{\{(g,h),(g,h')\}: g\in V(G),\ \{h,h'\}\in E(H)\bigr\},\\
E_{\mathrm{conn}}
&=\bigl\{\{(g,f(g')),(g',f(g))\}: \{g,g'\}\in E(G)\bigr\}.
\end{align*}
Using the standard symmetric definition,
\begin{equation}
    M(G\otimes_f H;x,y)=\sum_{\{u,v\}\in E(G\otimes_f H)}
x^{\min\{\deg_{G\otimes_f H}(u),\deg_{G\otimes_f H}(v)\}}\,
y^{\max\{\deg_{G\otimes_f H}(u),\deg_{G\otimes_f H}(v)\}}\,.
\end{equation}
Since $E(G\otimes_f H)$ is a disjoint union of $E_{\mathrm{inner}}$ and $E_{\mathrm{conn}}$,
we have
\[
\begin{aligned}
    M(G\otimes_f H;x,y)&=
\sum_{\{u,v\}\in E_{\mathrm{inner}}} x^{\min\{\deg_{G\otimes_f H}(u),\deg_{G\otimes_f H}(v)\}}y^{\max\{\deg_{G\otimes_f H}(u),\deg_{G\otimes_f H}(v)\}}\\
 &+ 
\sum_{\{u,v\}\in E_{\mathrm{conn}}} x^{\min\{\deg_{G\otimes_f H}(u),\deg_{G\otimes_f H}(v)\}}y^{\max\{\deg_{G\otimes_f H}(u),\deg_{G\otimes_f H}(v)\}}.
\end{aligned}
\]
Recall that degrees in $G\otimes_f H$ are given in~\eqref{lem:DegreeSierpinski}.  
An inner edge has the form $\{(g,h),(g,h')\}$ with $\{h,h'\}\in E(H)$. Hence its monomial
contribution equals
\[
x^{\min\{d_{g,h},d_{g,h'}\}}\,y^{\max\{d_{g,h},d_{g,h'}\}}.
\]
Summing over all $g\in V(G)$ and all $\{h,h'\}\in E(H)$ gives $M^{(\mathrm{inner})}(x,y)$.
 
A connecting edge corresponding to $\{g,g'\}\in E(G)$ has the form
$\{(g,f(g')),(g',f(g))\}$. Hence
the monomial contribution equals
\[
x^{\min\{d_{g,f(g')},\,d_{g',f(g)}\}}\,
y^{\max\{d_{g,f(g')},\,d_{g',f(g)}\}}.
\]
Summing over all $\{g,g'\}\in E(G)$ gives $M^{(\mathrm{conn})}(x,y)$.

Therefore, $M(G\otimes_f H;x,y)=M^{(\mathrm{inner})}(x,y)+M^{(\mathrm{conn})}(x,y)$.
\end{proof}

\section{Products of Paths}
\label{sec:Paths}

To illustrate formulas from the previous section, we consider here products of paths. Unless stated otherwise, we assume that paths considered are of order at least three. For the path $P_n$ ($n\ge 3$) itself we clearly have 
\begin{align*}
D_{P_n}(t) & = 2t+(n-2)t^2\,, \\
M(P_n;x,y) & = 2xy^2+(n-3)x^2y^2\,.
\end{align*}
We consider products of paths in the same order as in the previous section. 
\begin{itemize}

\item\textbf{Cartesian product}. For the grid graphs $P_m\cp P_n$, using Theorem~\ref{thm:cartesian}, we obtain 
\[
\begin{aligned}
M(P_m\cp P_n;x,y)
&=
\bigl(2xy+(m-2)(xy)^2\bigr)\bigl(2xy^2+(n-3)x^2y^2\bigr)\\
&\quad+
\bigl(2xy+(n-2)(xy)^2\bigr)\bigl(2xy^2+(m-3)x^2y^2\bigr)\,.
\end{aligned}
\] 

\item\textbf{Direct product}. Since $m_{1,2}(P_m)=2$ and $m_{2,2}(P_m)=m-3$, 
Theorem~\ref{thm:tensor-minmax} yields
\[
M(P_m\times P_n;x,y)
=
4x y^{4}
+
4x^{2}y^{2}
+
4(m+n-6)x^{2}y^{4}
+
2(m-3)(n-3)x^{4}y^{4}.
\]

\item\textbf{Strong product}. For the king graphs $P_m\boxtimes P_n$, using Theorem~\ref{thm:strong} we obtain 
\[
\begin{aligned}
M(P_m\boxtimes P_n;x,y)
&=
8x^3y^5
+
4x^3y^8
+
2(m+n-4)x^5y^5+
(6m+6n-32)x^5y^8 \\
& 
+
(4mn-11m-11n+30)x^8y^8.
\end{aligned}
\]

\item\textbf{Lexicographic product}. By Corollary~\ref{cor:lexico-compact}, we have  
\[
\begin{aligned}
M(P_m[P_n];x,y)
&=\bigl(2xy^2+(n-3)x^2y^2\bigr)\Bigl(2(xy)^n+(m-2)(xy)^{2n}\Bigr)\\
&\quad+\Bigl(2x^ny^{2n}+(m-3)x^{2n}y^{2n}\Bigr)\Bigl(2x+(n-2)x^2\Bigr)\Bigl(2y+(n-2)y^2\Bigr).
\end{aligned}
\]
Thus
\[
\begin{aligned}
M(P_m[P_n];x,y)
&=4x^{n+1}y^{n+2}
+2(m-2)x^{2n+1}y^{2n+2}
+2(n-3)x^{n+2}y^{n+2} \\
&\quad
+(n-3)(m-2)x^{2n+2}y^{2n+2} \\
&\quad
+8x^{n+1}y^{2n+1}
+4(n-2)x^{n+1}y^{2n+2}
+4(n-2)x^{n+2}y^{2n+1} \\
&\quad
+2(n-2)^2x^{n+2}y^{2n+2} \\
&\quad
+4(m-3)x^{2n+1}y^{2n+1}
+2(m-3)(n-2)x^{2n+1}y^{2n+2} \\
&\quad
+2(m-3)(n-2)x^{2n+2}y^{2n+1} \\
&\quad
+(m-3)(n-2)^2x^{2n+2}y^{2n+2}.
\end{aligned}
\]

\item\textbf{Symmetric-Difference Product}. 
By Theorem \ref{thm:MP-XOR}, we have 
\[
\begin{aligned}
M(P_m\oplus P_n;x,y) & = 
8x^{A}y^{B}+8x^{A}y^{C}+4(m+n-6)x^{A}y^{D}\\
&\quad  +4(m+n-6) x^{\min\{B,C\}}y^{\max\{B,C\}}\\
&\quad  +\Bigl(2n^2-12n+20+4(m-3)(n-3)\Bigr) x^{\min\{B,D\}}y^{\max\{B,D\}}\\
&\quad  +\Bigl(2m^2-12m+20+4(m-3)(n-3)\Bigr) x^{\min\{C,D\}}y^{\max\{C,D\}}\\
&\quad  +4(m-3)x^{C}y^{C}+4(n-3)x^{B}y^{B}\\
&\quad  +\Bigl((m-3)(n^2-6n+10)+(n-3)(m^2-6m+10)\Bigr) x^{D}y^{D},
\end{aligned}
\]
with $A:=\delta_\oplus^{P_m\oplus P_n}(1,1)=m+n-2$, $B:=\delta_\oplus^{P_m\oplus P_n}(1,2)=n+2m-4$, $C:=\delta_\oplus^{P_m\oplus P_n}(2,1)=m+2n-4$, and $D:=\delta_\oplus^{P_m\oplus P_n}(2,2)=2m+2n-8$.

\item\textbf{Disjunction Product}. By Theorem \ref{prop:M-disjunction}, we have 
\[
\begin{aligned}
M(P_m\vee P_n;x,y)= &
8x^{A}y^{B}+8x^{A}y^{C}
+4(m+n-5)x^{A}y^{D}
\\
&+4(m+n-5)x^{\min\{B,C\}}y^{\max\{B,C\}}\\
& +\bigl(4mn-12m+2n^2-20n+44\bigr)x^{B}y^{D}
\\
&+\bigl(2m^2+4mn-20m-12n+44\bigr)x^{C}y^{D}\\
&  +4(n-3)x^{B}y^{B}+4(m-3)x^{C}y^{C}\\
& +\bigl(m^2n+mn^2-10mn-3m^2-3n^2+22m+22n-42\bigr)x^{D}y^{D},
\end{aligned}
\]
with $A:=\delta_\vee(1,1)=m+n-1$, $B:=\delta_\vee(1,2)=2m+n-2$, $C:=\delta_\vee(2,1)=m+2n-2$, and $D:=\delta_\vee(2,2)=2m+2n-4$.

\item\textbf{Sierpi\'{n}ski Product}. Using Theorem~\ref{thm:sierpinski-M}, we determine the M-polynomial of the Sierpi\'{n}ski product in two basic cases of the function $f$. 
\begin{enumerate}

\item \textit{Constant map.} Assume that $f(i)=1$ for all $i\in [m]$.  Then
\[
M(P_m\otimes_f P_n;x,y)
=
mxy^2+\bigl(m(n-3)+2\bigr)x^2y^2
+mx^2y^3+(m-3)x^3y^3.
\]
\item  \textit{Identity map:} Assume that $m\le n$ and $f(i)=i$ for all $i\in [m]$. Then
     \[
M(P_m\otimes_f P_n;x,y)
= a_{1,2} x y^{2}
+ a_{1,3} x y^{3}
+ a_{2,2} x^{2} y^{2}
+ a_{2,3} x^{2} y^{3}
+ a_{3,3} x^{3} y^{3},
\]
where the coefficients are given  by
\[
(a_{1,2},a_{1,3},a_{2,2},a_{2,3},a_{3,3})=
\begin{cases}
\bigl(2m-6, 4, m^{2}-7m+14, 4m-10, m-3\bigr), & n=m,\\[2mm]
\bigl(2m-4, 3, m^{2}-6m+10, 4m-8, m-2\bigr), & n=m+1,\\[2mm]
\bigl(2m-3, 2, mn-7m+9, 4m-7, m-2\bigr), & n\ge m+2.
\end{cases}
\] 
\end{enumerate}
\end{itemize}

\section{Conclusion and Perspectives}
\label{sec:Conclusion}

In this paper, we developed a systematic framework for computing the M-polynomial under graph products in which the vertex set of a product is the Cartesian product of the vertex sets of its factors. Using formulas for vertex degrees, we derived explicit expressions for the corresponding M-polynomials $M(G\ast H;x,y)$ and, whenever possible, obtained compact factorizations in terms of $M(G;x,y)$, $M(H;x,y)$, and the degree polynomials $D_G$ and $D_H$. In particular, the Cartesian and the lexicographic product admit elegant compact forms that reduce the global computation to substitutions and products of lower-dimensional polynomials.

As already mentioned, most papers on the M-polynomial to date focus on specific classes of graphs and specific chemical molecules. However, further research is certainly needed to gain a better understanding of this polynomial in more general terms. This includes investigations of coefficients, roots, and recursive structures. 

From the perspective of our research, it certainly makes sense to further investigate other graph operations. Local operations appear less promising from this perspective, since they typically change the M-polynomial only slightly. On the other hand, there are ``non-product" global operations, such as joins, hierarchical products, corona products, and blow-up graphs, that are interesting from the perspective of the M-polynomial. 

\section*{Data Availability Statement}
 
Data sharing is not applicable to this article as no new data were created or analyzed in this study.

\section*{Conflict of Interest Statement}

The authors declare no conflict of interest.

\section*{Acknowledgments}

Sandi Klav\v{z}ar was supported by the Slovenian Research and Innovation Agency (ARIS) under the grants P1-0297, N1-0285, N1-0355, N1-0431, J1-70045.  

\bibliographystyle{acm} 
\bibliography{biblio}

\end{document}